\documentclass[12pt]{article}
\usepackage[margin=1in]{geometry}
\usepackage{amsmath,amsthm,amssymb}
\usepackage{mathrsfs}
\usepackage{amsmath}
\usepackage{amssymb}
\usepackage{epstopdf}
\usepackage{mathtools}
\usepackage{graphicx}
\usepackage{subfig}
\usepackage[utf8]{inputenc}
\usepackage{pgf,tikz}
\usetikzlibrary{arrows}
\usepackage{tikz-cd}
\usepackage{amsfonts}
\usepackage[unicode]{hyperref}
\hypersetup{ unicode = true }
\usepackage[utf8]{inputenc}
\usepackage{fancyhdr}
\usepackage{lastpage}
\usepackage{array}
\usepackage[font=small,labelfont=bf]{caption}
\usepackage{subfig}
\usepackage{amssymb}
\usepackage{empheq}
\usepackage{bigdelim}
\usepackage{flexisym}
\usepackage{pifont}

\usepackage{etoolbox}
\patchcmd{\thebibliography}{\section*{\refname}}{}{}{}

\usetikzlibrary{arrows}
\tikzset{commutative diagrams/.cd,arrow style=tikz,diagrams={>=latex'}}

\usetikzlibrary{positioning}

\makeatletter
\newcommand*{\rom}[1]{\expandafter\@slowromancap\romannumeral #1@}
\makeatother

\theoremstyle{plain}
\newtheorem{Th}{Theorem}[section]

\newtheorem{Claim}[Th]{Claim}
\theoremstyle{definition}

\newtheorem{Conj}[Th]{Conjecture}

\theoremstyle{plain}
\theoremstyle{remark}

\numberwithin{equation}{section}

\def\ex{\textrm{ex}}

\begin{document}

\title{Constructions of point-line arrangements in the plane with
large girth}

\author{Mozhgan Mirzaei\thanks{Department of Mathematics,  University of California at San Diego, La Jolla, CA, 92093 USA.  Supported by NSF grant DMS-1800736. Email:
{\tt momirzae@ucsd.edu}.}\and Andrew Suk\thanks{Department of Mathematics,  University of California at San Diego, La Jolla, CA, 92093 USA. Supported by an NSF CAREER award and an Alfred Sloan Fellowship. Email: {\tt asuk@ucsd.edu}.}\and Jacques Verstra\"ete\thanks{Department of Mathematics,  University of California at San Diego, La Jolla, CA, 92093 USA. Supported by NSF grant DMS-1800832. Email: {\tt jverstraete@ucsd.edu}.} }

\maketitle
\begin{abstract}
 A classical result by Erd\H{o}s, and later on by Bondy and Simonivits, states that every $n$-vertex graph with no cycle of length $2k$ has at most $O(n^{1+1 /k})$ edges. This bound is known to be tight when $k \in \{2,3,5\},$ but it is a major open problem in extremal graph theory to decide if this bound is tight for all $k$.

In this paper, we study the effect of forbidding short even cycles in incidence graphs of point-line arrangements in the plane.  It is not known if the Erd\H{o}s upper bound stated above can be improved to $o(n^{1+1/k})$ in this geometric setting, and in this note, we establish non-trivial lower bounds for this problem by modifying known constructions arising in finite geometries.  In particular, by modifying a construction due to Labeznik and Ustimenko, we construct an arrangement of $n$ points and $n$ lines in the plane, such that their incidence graph has girth at least $k + 5$, and determines at least $\Omega({n^{1+\frac{4}{k^2+6k-3}}})$ incidences.  We also apply the same technique to Wenger graphs, which gives a better lower bound for $k=5.$
\end{abstract}

\section{Introduction}
Let $k \geq 2$ be a fixed integer and let $C_{2k}$ denote a cycle of length $2k.$  If a graph contains cycles, then the length of the shortest cycle is called the \emph{girth} of the graph and is denoted by $g.$ A graph is $C_{2k}$-free if it contains no subgraph isomorphic to $C_{2k}.$ The \emph{Tur\'an number} $\ex\left(n, C_{2k}\right)$ denotes the maximum number of edges in a $C_{2k}$-free graph on $n$ vertices. An unpublished result of Erd\H{o}s, which was also proved by Bondy and Simonovits in \cite{bondy1974cycles}, shows that
\[
\ex\left(n, C_{2k}\right)=O\left(n^{1+1 /k}\right).
\]
In the other direction, one can use the probabilistic method to show that
$$\ex(n,C_{2k}) \geq \ex\left(n,\left\{C_{3}, C_{4}, \ldots, C_{2 k+1}\right\}\right) \geq  \Omega(n^{1+1/2k}).$$

\noindent For $k \in \{2,3,5\}$, it is known that $\ex\left(n, C_{2k}\right)=\Theta\left(n^{1+{1}/{k}}\right)$(see \cite{brown1966graphs,erd1966r} for $k=2$ and \cite{benson1966minimal, lazebnik1993new, wenger1991extremal} for $k\in \{3,5\}$). It is a long standing open problem to determine the order of magnitude of $\ex\left(n, {C}_{2k}\right)$ for $k \notin\{2,3,5\}.$  Erd\H{o}s and Simonovits \cite{erdHos1982compactness} conjectured that ${\ex}\left(n,{C}_{2k}\right)=\Theta\left(n^{1+1 /k}\right)$ for all $k \geq 2.$
For $k=4,$ the current best lower bound is due to Benson \cite{benson1966minimal} and Singleton \cite{singleton1966minimal}, who showed ${\ex}\left(n, {C}_{8} \right) \geq \Omega(n^{6 / 5})$.
For large $k,$ the densest known ${C}_{2k}$-free graphs on $n$ vertices are the constructions of Ramanujan graphs due to Margulis \cite{margulis1982explicit} Lubotzky, Phillips, and Sarnak \cite{lubotzky1988ramanujan},  Lazebnik, Ustimenko and Woldar \cite{lazebnik1995new}, and Dahan and Tillich \cite{dahan2014regular}. The constructions provide $n^{1+\delta}$ edges where $\delta\sim \frac{6}{7k}$ as $k \rightarrow \infty.$

\bigskip
In this paper, we study the analogous  problem for the point-line incidences in the plane.
Given a finite set $P$ of points in the plane and a finite set $\mathcal{L}$ of lines in the plane, let $I(P,\mathcal{L}) =\{(p, \ell) \in P\times \mathcal{L}:   p \in \ell \}$ be the set of incidences between $P$ and $\mathcal{L}$.  The \emph{incidence graph} of $(P,\mathcal{L})$ is the bipartite graph $G = (P\cup \mathcal{L}, I)$, with vertex parts $P$ and $\mathcal{L},$ and $E(G)=I(P, \mathcal{L})$. If $|P| = m$ and $|\mathcal{L}| = n$, then the celebrated theorem of Szemer\'edi and Trotter \cite{szemeredi1983extremal} states that
\begin{equation}\label{szt}
|I(P,\mathcal{L})|=O \left(m^{2/3}n^{2/3} + m + n\right).
\end{equation}
Moreover, this bound is tight which can be seen by taking the $\sqrt{m} \times \sqrt{m}$ integer lattice and bundles of parallel "rich" lines (see \cite{pach2011combinatorial}).  Thus, the number of incidences between any arrangement of $n$ points and $n$ lines is at most $O(n^{4/3})$, which is significantly better than $ex(n,C_4) = \Theta(n^{3/2})$.  Therefore, we make the following conjecture.

\begin{Conj}\label{conj1}
Let $k \geq 3$ be an integer.  Let $P$ be a set of $n$ points in the plane, let $\mathcal{L}$ be a set of $n$ lines in the plane, and $I =I(P,\mathcal{L})$.  If the incidence graph $G = (P\cup \mathcal{L},I)$ is $C_{2k}$-free, then $|I(P,L)|= o(n^{1+1/k}).$
\end{Conj}

\noindent Solymosi proved Conjecture \ref{conj1} for $k=3,$ but it is still an open problem for $k \geq 4.$

In this paper, we apply a simple technique of realizing point-line incidence graphs arising from finite geometries as point-line incidence graphs in the Euclidean plane.  As expected, we obtain some loss on the number of edges.  Our first result is stated below, which is obtained by applying this technique to a well-known construction of Lazebnik and Ustimenko \cite{lazebnik1995explicit} of a large graph with large girth.

\begin{Th}\label{Lower Bound 2k}
Let $k\geq 3$ be an odd integer. For $n >1,$ there exists an arrangement of $n$ points ${P}$ and $n$ lines $\mathcal{L}$ in the plane such that their incidence graph $G=(\mathcal{P} \cup \mathcal{L}, I)$ has girth $g \geq k+5$ and determines at least $\Omega({n^{1+\frac{4}{k^2+6k-3}}})$ incidences.
\end{Th}

For $k=5,$ we obtain a slightly better bound by applying this technique to Wenger graphs.


\begin{Th}\label{Lower Bound 10}
There exists an arrangement of $n$ points $\mathcal{P}$ and $n$ lines $\mathcal{L}$ in the plane such that their incidence graph $G=(\mathcal{P} \cup \mathcal{L}, I)$ is $C_{10}$-free and determines at least $\Omega({n^{1+{1}/{15}}})$ incidences.
\end{Th}

We systemically omit floor and ceiling signs whenever they are not crucial for the sake
of clarity of our presentation.

\section{Proof of Theorem \ref{Lower Bound 2k}}
In \cite{lazebnik1995explicit}, Lazebnik and Ustimenko gave the following construction of a $q$-regular bipartite graph $D(q,k)=(U \cup V, E)$ on $2 q^{k}$ vertices with girth $g \geq k+5,$ where $q$ is a prime power and $k \geq 3$ is an odd integer. The vertex set of $D(q,k)$ is $U \cup V,$ where $U=V=\{0,1, \ldots, q-1\}^k.$ In order to define $E \subset U \times V,$ we will label the coordinates of $u \in U$ and $v \in V$ as follows.
$$~~~u=\left(u_{1}, u_{1,1}, u_{1,2}, u_{2,1}, u_{2,2}, u_{2,2}^{\prime}, u_{2,3}, u_{3,2}, u_{3,3}\ldots, u_{i, i}^{\prime}, u_{i, i+1}, u_{i+1, i}, u_{i+1, i+1}, \ldots\right).$$

\noindent and
$$v=\left(v_{1}, v_{1,1}, v_{1,2}, v_{2,1}, v_{2,2}, v_{2,2}^{\prime}, v_{2,3}, v_{3,2}, v_{3,3}, \ldots, v_{i, i}^{\prime}, v_{i, i+1}, v_{i+1,i}, v_{i+1, i+1}, \ldots\right),$$

\noindent Note that we only consider the first $k$ such coordinates. For example when $k=5,$ we have 
$u=\left(u_{1}, u_{1,1}, u_{1,2}, u_{2,1}, u_{2,2}\right), v=\left(v_{1}, v_{1,1}, v_{1,2}, v_{2,1}, v_{2,2}\right).$
Then $uv \in E $ if and only if the following $k-1$ equations are satisfied:

\[
                        \begin{array}{ll}
v_{1,1}-u_{1,1}=v_{1}u_1 \hspace{2.3cm} \bmod q   \\
v_{1,2}-u_{1,2}=v_{1,1}u_1  \hspace{2.05cm} \bmod q \\
v_{2,1}-u_{2,1}=v_{1}u_{1,1}  \hspace{2.05cm} \bmod q \\ 
\end{array}\\
\]
\[
~~~~~~~~~~~~~~~~\left.\begin{array}{c}
v_{i, i}-u_{i, i}=v_{1} u_{i-1, i}  \hspace{2cm} \bmod q\\ {v_{i, i}^{\prime}-u_{i, i}^{\prime}=u_{1} v_{i, i-1}} \hspace{1.99cm} \bmod q  \\ {v_{i, i+1}-u_{i, i+1}=u_{1} v_{i, i}} \hspace{1.6cm} \bmod q \\ {v_{i+1, i}-u_{i+1, i}=v_{1} u_{i, i}^{\prime}} \hspace{1.6cm} \bmod q\end{array}\right\}
 \hspace{0.8cm}  i \geq 2. 
\]
where  $u_{1,1}^{\prime}=u_{1,1}, v_{1,1}^{\prime}=v_{1,1} . $ 

For example, if $k=5$, $uv \in E$ if and only if the following $4$ equations are satisfied.

\begin{align*}
~~~~~~~~~~~~~~~~~~\begin{array}{l}
v_{1, 1}-u_{1, 1}=v_{1} u_{1}  \hspace{2.35cm} \bmod q \\ 
{v_{1, 2}-u_{1, 2}= v_{1, 1}u_{1}} \hspace{2.1cm} \bmod q \\
 {v_{2,1}-u_{2,1}=v_{1} u_{1, 1}} \hspace{2.1cm} \bmod q \\
 {v_{2, 2}-u_{2, 2}}=v_{1} u_{1,2} \hspace{2.1cm} \bmod q . \hspace{2mm} ~~~~~~~~~~~~~~ 
 \end{array}
\end{align*}

\noindent In \cite{lazebnik1995explicit}, Lazebnik and Ustimenko proved the following.

\begin{Th}[Theorem 3.3 in \cite{lazebnik1995explicit}]\label{girth theorem}
For an odd integer $k \geq 3,$ the bipartite graph $D(k,q)$ described above has girth $g \geq k+5.$
\end{Th}

We will use this construction to prove Theorem \ref{Lower Bound 2k}.
\begin{proof}[Proof of Theorem \ref{Lower Bound 2k}]
Let $k \geq 3,$ and $n > 1$ be sufficiently large such that $\lfloor n^{\frac{4}{k^2+6k-3}} \rfloor \geq 1.$ By Bertrand's postulate theorem, there is a prime number $q$ such that $4{n}^{\frac{8}{k}}< q < 8{n}^{\frac{8}{k}}.$ Then let $D(q,k)=(U\cup V, E )$ be defined as above. Let $U' \subset U$ and $V' \subset V$ such that $u \in U'$ if

 \begin{align*}
~~~~~~~~~~~~~~~~\begin{array}{l}
0 \leq u_1 \leq n^{\frac{4}{k^2+6k-3}},\\
0 \leq u_{i,j} \leq n^{\frac{4(i+j)}{k^2+6k-3}} ~~~~~~    i,j \geq 1,\\
0 \leq u'_{i,j} \leq n^{\frac{4(i+j)}{k^2+6k-3}} ~~~~~~   i,j \geq 1.
 \end{array}
 \end{align*}
and $v \in V'$ if

\begin{align*}
~~~~~~~~~~~~~~\begin{array}{l}
0 \leq v_1 \leq 2n^{\frac{4}{k^2+6k-3}},\\
0 \leq v_{i,i+1} \leq 4n^{{\frac{4(2i+1)}{k^2+6k-3}}},\\
0 \leq v_{i+1,i} \leq 3n^{{\frac{4(2i+1)}{k^2+6k-3}}}, ~~~~~~~   i \geq 1,\\
0 \leq v'_{i,i} \leq 4n^{{\frac{4(2i)}{k^2+6k-3}}},\\
0 \leq v_{i,i} \leq 3n^{{\frac{4(2i)}{k^2+6k-3}}}.\\
                        \end{array}
 \end{align*}
 
 Then let $E' \subset U' \times V'$ such that $uv \in E'$ if and only if the following $k-1$ equations are satisfied

\[
                        \begin{array}{ll}
v_{1,1}-u_{1,1}=v_{1}u_1  \\
v_{1,2}-u_{1,2}=v_{1,1}u_1   \\
v_{2,1}-u_{2,1}=v_{1}u_{1,1}
\end{array}
\]
\[
~~~~~~~~~~~~~~~~~\left.\begin{array}{c}
v_{i, i}-u_{i, i}=v_{1} u_{i-1, i}  \label{eq:einstein} \tag{2.1}\\ {v_{i, i}^{\prime}-u_{i, i}^{\prime}=u_{1} v_{i, i-1}}  \\ 
~~~{v_{i, i+1}-u_{i, i+1}=u_{1} v_{i, i}} \\ 
~~~{v_{i+1, i}-u_{i+1, i}=v_{1} u_{i, i}^{\prime}}
  \end{array}\right\} \hspace{0.8cm}  i \geq 2. 
\]

The bipartite graph $G=(U' \cup V', E')$ is a subgraph of $D(q,k),$ and therefore, it has girth $g \geq k+5.$ It suffices to show that $|E'|= \Omega(n^{1+\frac{4}{k^2+6k-3}})$ and to realize $G$ as an incidence graph of points and lines in the plane.

Set $P=U' \subset \mathbb{Z}^k$. Then we have 
\begin{align*}
|P|&=n^{\frac{4}{k^2+6k-3}}\cdot n^{\frac{4(2)}{k^2+6k-3}} \cdot \left( \prod_{i=3}^{\frac{k-3}{2}-1} n^{\frac{4i}{k^2+6k-3}} \right)\cdot n^{\frac{4(\frac{k-3}{2})}{k^2+6k-3}}\\ 
&=n^{\frac{4}{k^2+6k-3}{\left(1+2+2(3)+ \ldots+2\left(\frac{k+3}{2}-1\right)+\frac{k+3}{2}\right)}}\\
& =n.
\end{align*}
We define a set of $|V'|$ lines $\mathcal{L}$ in $\mathbb{R}^k$ as follows. For each $v \in V',$ let $\ell_v$ be the solution space to the following system of $k-1$ equations over $k$ variables.
\newpage
\[
                        \begin{array}{ll}
v_{1,1}-x_{1,1}=v_{1}x_1   \\
v_{1,2}-x_{1,2}=v_{1,1}x_1 \\
\vspace{3mm}
v_{2,1}-x_{2,1}=v_{1}x_{1,1} \\
\end{array}
\]
\[ 
~~~~~~~~~~~~~~~~\left.\begin{array}{c}
v_{i,i}-x_{i,i}=v_{1}x_{i-1,i} \\
v'_{i,i}-x'_{i,i}=x_{1}v_{i,i-1} \label{eq:einstein2} \tag{2.2} \\
~~~v_{i,i+1}-x_{i,i+1}=x_{1}v_{i,i} \\
~~~v_{i+1,i}-x_{i+1,i}=v_{1}x'_{i,i} \\
  \end{array}\right\} \hspace{0.8cm}  i \geq 2. 
\]
Set $\mathcal{L}=\{\ell_v: v \in V'\}.$ It is easy to see that the $k-1$ equations above are independent and therefore, $\ell_v$ is a line in $\mathbb{R}^k.$ Moreover, each line is unique by the following claim.
\begin{Claim}
If $v,w \in V'$ are distinct, then $\ell_v \neq \ell_w.$
\end{Claim}
\begin{proof}
Let $v$ and $w$ be two distinct members of $V',$ where
\begin{align*}
~~~~~v&=(v_1,v_{1,1},v_{1,2},v_{2,1},v_{2,2},v'_{2,2},v_{2,3} ,v_{3,2},v_{3,3},\ldots, v'_{i,i}, v_{i,i+1},v_{i+1,i},v_{i+1,i+1}, \ldots),\\
 ~~~~~w&=(w_1,w_{1,1},w_{1,2},w_{2,1},w_{2,2},w'_{2,2},w_{2,3} ,w_{3,2},w_{3,3}, \ldots, w'_{i,i}, w_{i,i+1},w_{i+1,i},w_{i+1,i+1}, \ldots).
\end{align*}
Without loss of generality, we can assume $u_{1} \neq w_{1},$ as otherwise, we can show inductively that both vectors $u$ and $v$ have the same coordinates. For any point $u \in P,$
$$u=(u_1,u_{1,1},u_{1,2},u_{2,1},u_{2,2},u'_{2,2},u_{2,3},u_{3,2},u_{3,3} ,\ldots, u'_{i,i}, u_{i,i+1},u_{i+1,i},u_{i+1,i+1}, \ldots),$$
there is a unique solution to the following system of $k$ equations.

\[
 \begin{array}{ll}
w_{1,1}-u_{1,1}=w_{1}u_1\\                       
v_{1,1}-u_{1,1}=v_{1}u_1\\
v_{1,2}-u_{1,2}=v_{1,1}u_1
\end{array}
\]   
\[
~~~~~~~~~~~~~~~~~~\left. \begin{array}{c}
v_{i,i}-u_{i,i}=v_{1}u_{i-1,i}  \\
v'_{i,i}-u'_{i,i}=u_{1}v_{i,i-1}\\
~~~v_{i,i+1}-u_{i,i+1}=u_{1}v_{i,i}\\
~~~v_{i+1,i}-u_{i+1,i}=v_{1}u'_{i,i} 
\end{array}\right\} \hspace{1cm}  i \geq 2 .
\]                                      
 Hence, both $v$ and $w$ correspond to distinct lines.   \end{proof}

Therefore, we have
\begin{align*}
|\mathcal{L}|&=\Theta\left(n^{\frac{4}{k^2+6k-3}}\cdot n^{\frac{4(2)}{k^2+6k-3}} \cdot \left( \prod_{i=3}^{\frac{k-3}{2}-1} n^{\frac{4i}{k^2+6k-3}} \right)\cdot n^{\frac{4(\frac{k-3}{2})}{k^2+6k-3}}\right)\\ 
&=\Theta \left(n^{{\frac{4}{k^2+6k-3}{\left(1+2+2(3)+ \ldots+2\left(\frac{k+3}{2}-1\right)+\frac{k+3}{2}\right)}}}\right)\\
& =\Theta(n).
\end{align*}

Notice that every point $u \in P$ is incident to at least $2n^{\frac{4}{k^2+6k-3}}$ lines in $\mathcal{L}.$ Indeed, consider the $k-1$ equations (2.2).  There are $2n^{\frac{4}{k^2+6k-3}}$ choices for $v_1$. For fixed $u \in U,$ by fixing $v_1$ and by equation (2.1), we obtain $v_{1,1}$ sch that $0 \leq v_{1,1}=v_1u_1+u_{1,1} \leq 3 n^{\frac{8}{k^2+6k-3}}.$ By repeating the same argument the rest of the coordinates of $v$ will be determined uniquely.

Hence, we have a set of $\Theta(n)$ points $P$ and $\Theta(n)$ lines in $\mathbb{R}^k$ such that $|I(P,k)|=\Omega(n^{1+\frac{4}{k^2+6k-3}}).$  Projecting points $P$ and lines $\mathcal{L}$ into the plane completes the proof of Theorem \ref{Lower Bound 2k}.\end{proof}


\section{Proof of Theorem \ref{Lower Bound 10}}
Let $k \in \{2,3,5\}.$ In \cite{wenger1991extremal}, Wenger gave the following construction of a $C_{2k}$-free and $p$-regular bipartite graph $H_{k}(p)=(U \cup V, E)$ on $2 p^{k}$ vertices, where $p$ is a prime power. First we briefly discuss the Wenger's construction. The vertex set of $H_{k}(p)$
is  $U \cup V,$ where $U=V=\{0, \ldots, p-1\}^k$.
In order to define $E \subset U \times V,$ we will label the coordinates of $u \in U$ and $v \in V$ as follows.

$$~~~~~~u=\left(u_{0}, u_{1}, \ldots, u_{k-1}\right) ~~~~~ \text {and } ~~~~~v=\left(v_{0}, v_{1}, \ldots, v_{k-1}\right).$$
Then $uv \in E$ if and only if the following $k-1$ equations are satisfied:

$$v_j=u_j+u_{j+1} v_{k-1} ~~~\bmod p \hspace{0.90cm}  j=0, \ldots, k-2.$$

In \cite{wenger1991extremal}, Wenger proved the following.
\begin{Th}\label{Cycle free Wenger}
For $k\in \{2,3,5\},$ the bipartite graph $H_{k}(p)$ described above is $C_{2k}$-free.
\end{Th}
We will use this construction to prove Theorem \ref{Lower Bound 10}.
\begin{proof}[Proof of Theorem \ref{Lower Bound 10}]

Let $k \geq \{2,3,5\}$ and $n > 1$ be sufficiently large such that $\lfloor n^{\frac{2}{k^2+2k}} \rfloor \geq 1.$ By Bertrand's postulate theorem, there is a prime number $p$ such that $2^{2k}{n}^{\frac{2}{k}}< p < 2^{2k+1}n^{\frac{2}{k}}.$ Then let $H_{k}(p)=(U \cup V, E)$ be defined as above. Let $U' \subset U$ and $V ' \subset V$ such that $u \in U'$ if 

\begin{align*}
0 &\leq u_i  \leq 2^{2(k-i-1)}n^{\frac{2(k-i)}{k(k+1)}} \hspace{1.5cm} 0 \leq i \leq k-2, \\
0 &\leq u_{i} \leq n^{\frac{2}{k(k+1)}} \hspace{3cm} i=k-1,\\
\end{align*}

\noindent and $v \in V'$ if 

\begin{align*}
 2^{2(k-i-1)-1}n^{\frac{2(k-i)}{k(k+1)}}  &\leq v_i  \leq 2^{2(k-i-1)}n^{\frac{2(k-i)}{k(k+1)}} \hspace{1cm} 0 \leq i \leq k-2,\\
                         n^{\frac{2}{k(k+1)}} &\leq v_{i} \leq 2n^{\frac{2}{k(k+1)}} \hspace{2.6cm} i=k-1.
\end{align*}

Let $E' \subset U' \times V'$ such that $uv \in E'$ if and only if the following $k-1$ equations are satisfied:
\begin{align}\label{nomod}
~~~~~v_j=u_j+u_{j+1} v_{k-1}  \hspace{2.6cm}  j=0, \ldots, k-2.
\end{align}

The bipartite graph $G=(U' \cup V', E')$ is a subgraph of $H_k(p),$ and therefore, it is $C_{2k}$-free. It suffices to show that $|E'|\geq \Omega(n^{1+\frac{2}{k^2+k}})$ and to realize $G$ as the incidence graph of points and lines in the plane.

Let $P=U' \subset \mathbb{Z}^k.$ Then we have 

\begin{align*}
|P|&=\prod_{i=0}^{k-1} 2^{2(k-i-1)}n^{\frac{2(k-i)}{k(k+1)}}=2^{k(k-1)} n=\Theta_k(n).
\end{align*}

We define a set of $|V'|$ lines $\mathcal{L}$ in $\mathbb{R}^k$ as follows. For each $v \in V',$ let $\ell_v$ be the solution space to the following system of $k-1$ equations over $k$ variables.

\begin{align}\label{relations2}
x_i+v_{k-1}x_{i+1}-v_i=0  \hspace{2.2cm}  j=0, \ldots, k-2.
\end{align}
Set $\mathcal{L}=\{\ell_v: v \in V'\}.$ It is easy to see that the $k-1$ equations above are independent and therefore, $\ell_v$ is a line in $\mathbb{R}^k.$ Moreover, we have the following claim.
\begin{Claim}
If $v,w \in V'$ are distinct, then $\ell_v \neq \ell_w.$
\end{Claim}
\begin{proof}
Let $v$ and $w$ be two distinct members of $V',$ where
\begin{align*}
 v=(v_1, \ldots,v_{k-1} )~~~~ \text {and } ~~~~w=(w_1, \ldots,w_{k-1} ).
\end{align*}
Without loss of generality, we can assume $u_{1} \neq w_{1},$ as otherwise, we can show inductively that both vectors $u$ and $v$ have the same coordinates. For any point $u=(u_1,u_2, \ldots, u_{k-1}) \in P,$
there is a unique solution to the following system of $k$ equations.
\begin{align*}
                        \begin{array}{ll}
 x_1+w_{k-1}x_{2}-w_1=0\\                       
 x_i+v_{k-1}x_{i+1}-v_1=0 \hspace{3cm} 0 \leq i \leq k-2.\\
                        \end{array}
\end{align*}
Hence, both $v$ and $w$ correspond to distinct lines.   \end{proof}

Therefore, we have

\begin{align*}
|\mathcal{L}|&=\frac{1}{{2^{k-1}}}{\prod_{i=0}^{k-1} 2^{2(k-i-1)}n^{\frac{2(k-i)}{k(k+1)}}}=2^{k^2-1}n=\Theta_k(n).
\end{align*}

Notice that every line $\ell_v \in \mathcal{L}$ is incident to at least $n^{\frac{2}{k^2+k}}$ points in $P.$ Indeed, consider $k-1$ equations \ref{relations2}. There are $n^{\frac{2}{k^2+k}}$ choices for $u_{k-1}$. For fixed $ v \in V',$ by fixing $u_{k-1}$ and by equation \ref{nomod}, we obtain $u_{k-2}=v_{k-2}-u_{k-1}v_{k-1},$ such that

 $$0 \leq u_{k-2}  \leq 2^2n^{\frac{4}{k(k+1)}},$$

\noindent since

$$ 2 n^{\frac{4}{k(k+1)}} \leq v_{k-2} \leq 2^2 n^{\frac{4}{k(k+1)}} ~~~~\text{ and }~~~ -2n^{\frac{4}{k(k+1)}}\leq -u_{k-1}v_{k-1} \leq 0.$$

\noindent  By repeating the same argument the other coordinates of $u$ will be determined uniquely.

Thus we have a set of $\Theta(n)$ points $P$ and $\Theta(n)$ lines in $\mathbb{R}^k$ such that $|I(P,k)|=\Omega(n^{1+\frac{2}{k^2+k}}).$  Projecting points $P$ and lines $\mathcal{L}$ into the plane completes the proof of Theorem.\end{proof}


\section{ Concluding Remarks}
As a corollary of Theorem \ref{Lower Bound 2k} for $k=3,$ we get that there exists an arrangement of $n$ points and $n$ lines in the plane with no $C_6$ in the incidence graph while determining $\Omega(n^{1+\frac{1}{6}})$ incidences. It is worth mentioning that one can follow the construction of Theorem 1.5 in \cite{mirzaei2018grids} to get a construction of an arrangement of $n$ points and $n$ lines in the plane where their incidence graph is $C_6$-free while it determines $\Omega(n^{1+\frac{1}{7}})$ incidences.

\section{References}
\bibliographystyle{acm}
\bibliography{bibfile}

\end{document}